\documentclass[11pt]{amsart}
\usepackage{amssymb,amsfonts,amsthm,amscd,stmaryrd,dsfont,esint,upgreek}
\usepackage[numbers]{natbib}
\usepackage[mathcal]{euscript}
\usepackage[top=1.25in,bottom=1.25in, left=1in,right=1in]{geometry}
\theoremstyle{plain}
\newtheorem{thm}{Theorem}
\newtheorem{cor}{Corollary}
\newtheorem{lem}[cor]{Lemma}
\newtheorem{prop}[cor]{Proposition}

\theoremstyle{definition}

\newtheorem{remark}[cor]{Remark}

\numberwithin{cor}{section}
\numberwithin{equation}{section}

\newcommand{\R}{\mathbb{R}}

\newcommand{\N}{\mathbb{N}}
\renewcommand{\d}{d}
\newcommand{\Rd}{\mathbb R^\d}
\newcommand{\ep}{\varepsilon}

\newcommand{\E}{\mathbb{E}}
\newcommand{\Prob}{\mathbb{P}}

\newcommand{\Lip}{{\mathcal{L}}}
\newcommand{\SL}{\mathcal{S}}
\newcommand{\SLp}{\SL^+}
\newcommand{\SLm}{\SL^-}

\DeclareMathOperator{\USC}{USC}
\DeclareMathOperator{\LSC}{LSC}

\DeclareMathOperator{\dist}{dist}

\DeclareMathOperator*{\esssup}{ess\,sup}
\DeclareMathOperator*{\essinf}{ess\,inf}

\begin{document}

\title[Stochastic homogenization of level-set convex Hamilton-Jacobi equations]{Stochastic homogenization of level-set convex Hamilton-Jacobi equations}

\author[S. N. Armstrong]{Scott N. Armstrong}
\address{Department of Mathematics\\ University of Wisconsin\\ 480 Lincoln Drive\\
Madison, Wisconsin 53706.}
\email{armstron@math.wisc.edu}
\author[P. E. Souganidis]{Panagiotis E. Souganidis}
\address{Department of Mathematics\\ The University of Chicago\\ 5734 S. University Avenue
Chicago, Illinois 60637.}
\email{souganidis@math.uchicago.edu}
\date{\today}
\keywords{stochastic homogenization, Hamilton-Jacobi equation, correctors, first-passage percolation, effective Hamiltonian}
\subjclass[2010]{35B27}

\begin{abstract}
We present a simple new proof for the stochastic homogenization of quasiconvex (level-set convex) Hamilton-Jacobi equations set in stationary ergodic environments. Our approach, which is new even in the convex case, yields more information about the qualitative behavior of the effective nonlinearity.
\end{abstract}

\maketitle

\section{Introduction} \label{I}

In this paper we give a new proof for the homogenization, in  stationary ergodic settings, of level-set convex Hamilton-Jacobi equations. Specifically, we study the behavior, as $\ep \to 0$, of viscosity solutions of the stationary equation
\begin{equation}\label{HJ}
u^\ep + H\left(Du^\ep, \frac x\ep,\omega \right) = 0 \quad \mbox{in} \ \Rd.
\end{equation}
The Hamiltonian $H=H(p,y,\omega)$ is a random process since it depends on $\omega$, an element of an underlying probability space $(\Omega, \mathcal F,\Prob)$. We postpone the precise hypotheses to the next section, but we remark here that the Hamiltonian $H=H(p,y,\omega)$ is level-set convex and coercive in $p$ and stationary-ergodic in $(y,\omega)$.

Our main result is stated as follows.

\begin{thm} \label{H}
Assume \eqref{preserve}-\eqref{sqc}. There exists a level-set convex and coercive $\overline H\in C(\Rd)$ and an event $\Omega_0 \subseteq \Omega$ of full probability such that, for each $\omega\in \Omega_0$, the unique solution $u^\ep=u^\ep(\cdot,\omega)\in C(\Rd)$ of \eqref{HJ} converges locally uniformly in $\Rd$, as $\ep \to 0$, to the unique (constant) solution $u$ of
\begin{equation}\label{HJh}
u + \overline H(Du) = 0 \quad \mbox{in} \ \Rd.
\end{equation}
\end{thm}

The homogenization of~\eqref{HJ} for convex Hamiltonians $H$ was first proved by Souganidis~\cite{S} and Rezakhanlous and Tarver~\cite{RT}. Theorem~\ref{H} was first proved by Lions and Souganidis~\cite{LS4} several years ago. Their unpublished proof, which requires only \eqref{qc} and not \eqref{Lambda}-\eqref{sqc}, is entirely different from ours and consisted of approximating level-set convex Hamiltonians by functions of the form $\varphi (G)$, with $\varphi$ monotone and $G$ convex. Finally, Theorem~\ref{H} in one space dimension was proved by Siconolfi and Davini~\cite{DS0} using different arguments than ours.

We state the result for the time-independent problem for ease of exposition. Indeed, it is well-known (see~\cite{ASo} for example) that Theorem~\ref{H} is sufficient to obtain the homogenization of the time-dependent problem
\begin{equation}\label{HJt}
\left\{ \begin{aligned} 
& u_t^\ep + H\left(Du^\ep, \frac x\ep, \omega\right) = 0 & \mbox{in} & \ \Rd \times (0,\infty), \\
& u^\ep(\cdot,0,\omega) = u_0 & \mbox{on} & \ \Rd.
\end{aligned} \right.
\end{equation}

The purpose of this paper is to present a complete proof of Theorem~\ref{H}, which does not rely, as in previous work, either on explicit formulae for $u^\ep$ or on the existence of appropriate subcorrectors. The argument, which, as already mentioned, is new even for convex $H$, consists of homogenizing the equation ``one level-set of $\overline H$ at a time." Thus the generality of level-set convexity is quite natural. 

More precisely, we consider the eikonal-type equations (also called the \emph{metric problems})
\begin{equation}\label{eik-pre}
H(Du,y,\omega) = \mu \quad \mbox{in} \ \Rd \setminus \{ x \}, \quad u(x) = 0,
\end{equation}
and
\begin{equation}\label{eik-erp}
H(-Dv,y,\omega) = \mu \quad \mbox{in} \ \Rd \setminus \{ x \}, \quad v(x) = 0.
\end{equation}
It turns out that, for $\mu\in \Rd$ larger than some critical value, \eqref{eik-pre} and \eqref{eik-erp} have unique maximal solutions $m_\mu(\cdot,x,\omega)$ and $n_\mu(\cdot,x,\omega)$, respectively, satisfying $m_\mu(y,x,\omega)=m_\mu(x,y,\omega)$, which are necessarily stationary in the appropriate way and subadditive. The subadditive ergodic theorem yields almost sure limits, along rays from the origin, for the normalized $m_\mu$ and $n_\mu$, i.e., the existence of deterministic functions $\overline m_\mu, \ \overline n_\mu:\Rd \to \R$ such that, as $t\to \infty$, almost surely in $\omega$ and locally uniformly in $y\in\Rd$,
\begin{equation*}\label{}
t^{-1} m_\mu(ty,0,\omega) \rightarrow \overline m_\mu(y) \quad \mbox{and} \quad t^{-1} n_\mu(ty,0,\omega) \rightarrow \overline n_\mu(y).
\end{equation*}
It turns out, as we show here, that the above limits, which can be rephrased as a homogenization result for the metric problems, suffice to prove Theorem~\ref{H}. 

The role of the maximal solutions $m_\mu$ in the general theory of viscosity solutions of convex Hamilton-Jacobi equations has been known for some time (see, e.g., Lions~\cite{Li}). The first proofs of the stochastic homogenization of~\eqref{HJ} was based on the behavior of the time-dependent versions of \eqref{eik-pre}. The metric problem \eqref{eik-pre} for convex $H$ in the stationary ergodic setting was analyzed in detail in~\cite{DS0,DS1,DS2} without however showing that its weighted asymptotic implies the homogenization of \eqref{HJ}. The authors~\cite{ASo} introduced the metric problem for general degenerate elliptic quasilinear viscous Hamilton-Jacobi equations and showed that its asymptotic behavior together with the existence of appropriate subcorrectors yields homogenization. The metric problem for certain uniformly elliptic viscous Hamilton-Jacobi equations was previously considered by Sznitman in his study of the large deviations of a diffusion in random environment (see his book~\cite{Szb}). 

Here we show that by considering both \eqref{eik-pre} and \eqref{eik-erp} it is possible to obtain homogenization of level-set convex equations without first having the existence of subcorrectors. The idea is to exploit the fact that, for fixed $p$ and $\mu=\overline H(p)$ and $|x|$ large, $m_\mu(y,x,\omega) -p\cdot y$ is an approximate supercorrector (as a function of $y$) and $-n_\mu(x,y,\omega)- p\cdot y$ is an approximate subcorrector in the ball $B(0, |x|/4)$. The argument is related to the approach the authors introduced in~\cite{ASo} to homogenize viscous Hamilton-Jacobi equations in unbounded environments (see also~\cite{ASo2}). However, unlike the latter, it does not rely on the method proposed by Lions and Souganidis~\cite{LS3} based on the construction of a subcorrector with stationary, mean-zero gradient. In fact, we obtain the existence of a subcorrector (see Propostion~\ref{usual}) but this actually requires use of the homogenization theorem. In addition to the easy generalization to level-set convex $H$, the approach we put forward here is more constructive.

We also present some new results on the qualitative behavior of the effective Hamiltonian. In particular, in the case that $H$ is strictly convex in $p$, we generalize and give a new and simpler proof of a result of Evans and Gomes~\cite{EG} on the strict convexity of $\overline H$ in directions transverse to its level sets. This justifies the intuition that, if $H$ is strictly convex in $p$, then $\overline H$ can only fail to be strictly convex by having flat spots or flat edges in its level sets. 

Our methods extend to Hamiltonians with macroscopic spacial dependence, i.e., to
\begin{equation}\label{HJsd}
u^\ep + H\left( Du^\ep, x, \frac x\ep, \omega\right) = 0 \quad \mbox{in} \ \Rd
\end{equation}
and its time-dependent analogues. Let us mention, however, that our arguments seem to us to be strongly first-order in that they do not obviously apply to \emph{viscous} Hamilton-Jacobi equations like
\begin{equation*}\label{}
u^\ep - \Delta u^\ep + H\left( Du^\ep, x, \frac x\ep, \omega\right) = 0 \quad \mbox{in} \ \Rd.
\end{equation*}

The study of the homogenization of Hamilton-Jacobi equations was initiated in the unpublished work of Lions, Papanicolaou and Varadhan~\cite{LPV}, who obtained results in the periodic case. This was simplified by Evans~\cite{E} and extended (via a new proof) to the almost periodic case by Ishii~\cite{I}. As mentioned above, results for stochastic homogenization of first-order Hamilton-Jacobi equations were first obtained by Souganidis \cite{S} and  Rezakhanlou and Tarver \cite{RT} and, for viscous Hamilton-Jacobi equations, by Lions and Souganidis \cite{LS2} and Kosygina, Rezakhanlou, and Varadhan \cite{KRV}. We also mention the work of Kosygina and Varadhan \cite{KV} and Schwab \cite{Sch} who obtained homogenization results for these equations in spatio-temporal media.

In the next section we introduce the assumptions and give some preliminary results needed for the main argument. In Sections~\ref{EIK} and~\ref{MP} we review the eikonal equation in non-random and random environments, respectively. We give the proof of the main result in Section~\ref{HP}, and conclude with some qualititative properties of the effective equation in Section~\ref{C}.

\section{Preliminaries} \label{P}

We introduce the assumptions and present some preliminary results to be used later in the paper.

\subsection*{The hypotheses}
We consider a probability space $(\Omega, \mathcal F, \mathds P)$ endowed with an ergodic group $(\tau_y)_{y\in \Rd}$ of $\mathcal F$-measurable, measure-preserving transformations $\tau_y:\Omega\to \Omega$. That is, we assume that, for every $x,y\in\Rd$ and $A\in \mathcal F$,
\begin{equation}\label{preserve}
\tau_{x+y} = \tau_{x} \circ \tau_y \quad \mbox{and} \quad \Prob[\tau_y(A)] = \Prob[A],
\end{equation}
and
\begin{equation}\label{ergo}
\mbox{if} \ \ \tau_z(A) = A \ \ \mbox{for every} \ z\in \Rd, \ \ \mbox{then either} \ \ \Prob[A]=0 \ \mbox{or} \ \Prob[A]=1.
\end{equation}
The Hamiltonian $H:\Rd\times\Rd\times \Omega\to \R$ is assumed to be measurable with respect to the $\sigma$-algebra generated by $\mathcal B \times \mathcal B \times \mathcal F$, where $\mathcal B$ denotes the Borel $\sigma$-algebra of $\Rd$. We write $H = H(p,y,\omega)$, and we assume that $H$ is \emph{stationary} in $(y,\omega)$ with respect to the group $( \tau_y )_{y\in\Rd}$, that is, for every 
$p,y,z\in\Rd$ and $\omega \in \Omega$,
\begin{equation} \label{stationary}
H(p,y,\tau_z \omega) = H(p,y+z,\omega).
\end{equation}
As far as regularity in $(p,y)$, we require that, for every $R> 0$, the family of functions
\begin{equation}\label{reg}
\{ H(\cdot,\cdot,\omega) : \omega\in\Omega \} \quad \mbox{is bounded and equicontinuous on} \ B_R\times \Rd.
\end{equation}
We also assume that $H$ is \emph{coercive} in $p$ uniformly in $(y,\omega)$ in the sense that
\begin{equation}\label{coer}
\lim_{|p|\to\infty} \essinf_{(y,\omega) \in \Rd \times \Omega} H(p,y,\omega) = +\infty,
\end{equation}
and \emph{level-set convex} in $p$, i.e., for every $(y,\omega) \in \Rd \times \Omega$ and $p,q \in \Rd$,
\begin{equation}\label{qc}
H\left(\tfrac12 (p + q),y,\omega\right) \leq \max\left\{ H(p,y,\omega), H(q,y,\omega)\right\}.
\end{equation}
Of course, \eqref{qc} is equivalent to the statement that, for all $0 \leq \lambda \leq 1$, $p,q,y \in \Rd$ and $\omega \in \Omega$,
\begin{equation*}\label{}
H\left( \lambda p + (1-\lambda)q,y,\omega \right) \leq  \max\left\{ H(p,y,\omega), H(q,y,\omega)\right\}.
\end{equation*}
Actually, we need slightly more than \eqref{qc}, which is, roughly speaking, the requirement that, except for the minimal level set, the level sets of $H$ in $p$ have empty interior uniformly in $(y,\omega)$. We quantify this precisely by assuming there exists $\Lambda\in C(\R\times\R)$ satisfying
\begin{equation}\label{Lambda}
\Lambda \ \ \mbox{is nondecreasing in both of its arguments and, for all}   \  \mu \neq \nu, \ \Lambda(\mu,\nu) < \max\{ \mu,\nu\}
\end{equation}
and such that 
%
for all $p,q,y \in\Rd$ and $\omega\in \Omega$,
\begin{equation}\label{sqc}
H\left(\tfrac12 (p + q),y,\omega\right) \leq \Lambda\big( H(p,y,\omega), H(q,y,\omega)\big).
\end{equation}
Notice that $H$ is convex if and only if \eqref{sqc} holds with $\Lambda (\mu,\nu) = \frac12 (\mu + \nu)$. We remark $H(p,y,\omega):= \Phi(G(p,y,\omega))$ satisfies \eqref{qc} (respectively, \eqref{Lambda}-\eqref{sqc}, with $\Phi(\Lambda)$ in place of $\Lambda$) if $G$ is convex if $p$ and $\Phi:\R \to \R$ is increasing (respectively, strictly increasing). Another explicit example is $H(p):= |p|^{\gamma}$ with $\gamma > 0$, which satisfies \eqref{sqc} for $\Lambda(\mu,\nu) = \left( \frac12 (\mu^{1/\gamma} + \nu^{1/\gamma})\right)^\gamma$. 

It is easy to show by repeatedly applying \eqref{sqc} that, for each $0<\lambda \leq \frac12$, there exists a function $\Lambda_{\lambda}\in C(\R\times\R)$ satisfying \eqref{Lambda} such that, for every $p,q,y\in \Rd$ and $\omega\in \Omega$,
\begin{equation}\label{sqclambda}
H\left( \lambda p + (1-\lambda)q,y,\omega \right) \leq \Lambda_{\lambda} \big( H(p,y,\omega), H(q,y,\omega)\big)
\end{equation}
and the $\Lambda_\lambda$'s satisfy $\Lambda_{\lambda_1} \geq \Lambda_{\lambda_2}$ for each $0<\lambda_1 \leq \lambda_2 \leq \frac12$.

We emphasize that, throughout the paper, the hypotheses above, including \eqref{preserve}-\eqref{sqc}, are in force. In certain settings in which the probability space plays no important role (e.g., in Lemma~\ref{convtrick}, Remark~\ref{yeswecan} and the entirety of Section~\ref{EIK}), we drop the dependence on $\omega$, keeping the obvious analogues of \eqref{reg}, \eqref{coer} and \eqref{sqc} in place.

\subsection*{Some preliminary results}
We denote by $\USC(U)$ and $\LSC(U)$, respectively, the sets of real-valued upper and lower semicontinuous functions defined on $U$. The next lemma is standard, but because it is usually stated only for Hamiltonians convex in $p$, we include a proof for the convenience of the reader. Note that the statement is true for a level-set convex Hamiltonian only because the right side of \eqref{convteq} is constant. 

\begin{lem} \label{convtrick}
Fix $\nu\in \R$ and $U\subseteq \Rd$ open. Then $v\in \USC(U)$ is a viscosity solution of 
\begin{equation}\label{convteq}
H(Dv,y) \leq \nu \quad \mbox{in} \ U
\end{equation}
if and only if $v$ is locally Lipschitz in $U$ and satisfies \eqref{convteq} almost everywhere in $U$. 
\end{lem}
\begin{proof}
A viscosity solution of \eqref{convteq} is Lipschitz due to the coercivity of $H$, and therefore is differentiable almost everywhere in $U$ and satisfies \eqref{convteq} at any point of differentiability. This is classical and we refer to~\cite{Ba} for the details.

Conversely, suppose that $v$ is locally Lipschitz and satisfies \eqref{convteq} almost everywhere in $U$. Fix $\theta > 0$ and consider, for $\delta > 0$, the standard mollification $v^\delta (y):= \delta^{-\d} \int_{\Rd} v(y-z) \eta(z/\delta) \, dy$ of $v$ which is defined in $U_\delta:= \{ y\in U\, : \, \dist(y,\partial U) > \delta \}$. Then $v^\delta$ is smooth and, according to \eqref{reg} and \eqref{qc}, for $\delta>0$ sufficiently small, 
\begin{equation}\label{approx}
H(Dv^\delta,y) \leq \nu + \theta \quad \mbox{in} \ U_{2\delta}. 
\end{equation}
Since $v^\delta$ converges, as $\delta \to 0$, to $v$ locally uniformly in $U$, we may pass to limits $\delta \to 0$ in \eqref{approx} and then send $\theta \to 0$ to obtain \eqref{convteq} in the viscosity sense.
\end{proof}

It is useful to notice that \eqref{reg} and \eqref{coer} imply that $H$ is uniformly continuous ``from below" in the following sense (here we drop dependence on $\omega$): for every $\mu \in \R$ and $\alpha > 0$, there exists $\theta >0$, depending only on \eqref{reg}, \eqref{coer} and an upper bound for $\mu$, such that, for all $(p,y) \in \Rd \times \Rd$ for which $H(p,y)\geq \mu$,
\begin{equation}\label{ywc}
\inf_{|q| \leq \theta} H(p+q,y) \geq \mu - \alpha.
\end{equation}
If not, then there exists $\alpha_0>0$ and $p_n,q_n,y_n\in \Rd\times\Rd\times\Rd$ such that $q_n\to 0$ as $n\to \infty$, $H(p_n,y_n) \geq \mu$ and $H(p_n+q_n,y_n) \leq \mu-\alpha_0$.
According to \eqref{coer}, there exist $R>0$ such that $|p_n| \leq R$. Using~\eqref{reg} and the fact that $q_n\to 0$, we can choose $n$ large enough such that, for all $(p,y) \in B_R\times\Rd$,
\begin{equation*}\label{}
H(p+q_n,y) \geq H(p,y) - \frac12 \alpha_0.
\end{equation*}
Hence for large enough $n$ we obtain
\begin{equation*}\label{}
\mu - \alpha_0 \geq H(p_n+q_n,y_n) \geq H(p_n,y_n) - \frac12 \alpha_0 \geq \mu - \frac12 \alpha_0,
\end{equation*}
which is the desired contradiction. 

An immediate consequence of \eqref{ywc} is the following lemma, which states that a small, smooth perturbation of a supersolution is still nearly a supersolution.

\begin{lem}\label{yeswecan}
Fix $\nu\in \R$, $U\subseteq \Rd$ open and $\alpha> 0$. Suppose $v\in \LSC(U)$ is a viscosity solution of 
\begin{equation*}
H(Dv,y) \geq \mu \quad \mbox{in} \ U
\end{equation*}
and $\varphi$ is a smooth function with $\sup_{U} |D\varphi| \leq \theta$, where $\theta$ is as in~\eqref{ywc}.  Then $\widetilde v:= v + \varphi$ satisfies
\begin{equation*}\label{}
H(D\widetilde v,y) \geq \mu - \alpha \quad \mbox{in} \ U.
\end{equation*}
\end{lem}

A measurable function $\phi:\Omega\times \Rd\to \R$ is said to have \emph{stationary and mean-zero increments} if, for all $x,y,z\in \Rd$ and $\omega\in \Omega$,
\begin{equation*}\label{}
\phi(x,\tau_z\omega) - \phi(y,\tau_z\omega) = \phi(x+z,\omega) - \phi(y+z,\omega) \quad \mbox{and} \quad \E[ \phi(x,\cdot) ] = \E[\phi(0,\cdot)],
\end{equation*}
respectively. The following lemma is due to Kozlov~\cite{K} (see also~\cite{ASo} for a proof).

\begin{lem} \label{kozlov}
Suppose that $w:\Rd\times\Omega\to \R$ has stationary, mean-zero increments and is Lipschitz in its first variable, a.s. in~$\omega$. Then
\begin{equation}  \label{slinfty}
\lim_{|y|\to\infty} |y|^{-1}w(y,\omega) = 0 \quad \mbox{a.s. in} \ \omega.
\end{equation}
\end{lem}

We rely on a version of the subadditive ergodic theorem given in Akcoglu and Krengel~\cite{AK}. For the reader's convenience, we recall the precise statement, which requires some additional notation.

Denote by $\mathcal{I}$ the class of subsets of $[0,\infty)$ consisting of finite unions of intervals of the form~$[a,b)$ and let $(\sigma_t)_{t\geq 0}$ be a semigroup of measure-preserving transformations on $\Omega$. A \emph{continuous subadditive process} on $(\Omega, \mathcal F, \Prob)$ with respect to $(\sigma_t)_{t\geq 0}$ is a map $Q: \mathcal I \rightarrow L^1(\Omega,\Prob)$ such that
\begin{enumerate}
\item[(i)] $Q(I)(\sigma_t\omega) = Q( t+I)(\omega)$ for each $t>0$, $I \in \mathcal I$ and a.s. in $\omega$,
\item[(ii)] there exists a constant $C> 0$ such that for each $I\in \mathcal I$, $\E \big| Q(I) \big| \leq C |I|$ and
\item[(iii)] if $I_1,\ldots I_k \in \mathcal I$ are disjoint, then $Q(\cup_{j=1}^{k} I_j) \leq \sum_{j=1}^k Q(I_j)$.
\end{enumerate}

The statement of the subadditive ergodic theorem we need is the following. 

\begin{prop}[{\cite{AK}}] \label{SAET}
Suppose that $Q$ is a continuous subadditive process with respect to $(\sigma_t)_{t\geq 0}$. Then there exists a random variable $a$, which is invariant under $(\sigma_t)_{t\geq 0}$, such that
\begin{equation*} 
t^{-1} Q\big([0,t)\big)(\omega) \rightarrow a(\omega) \quad \mbox{a.s. in} \ \omega.
\end{equation*}
In particular, if $( \sigma_t )_{t>0}$ is ergodic, then $a$ is constant a.s. in $\omega$.
\end{prop}

\section{The eikonal equation in an exterior domain} \label{EIK}

In order to apply the subadditive ergodic theorem to show that \eqref{HJ} homogenizes, we must identify an appropriate subadditive quantity. This is accomplished by considering, for each fixed $x\in \Rd$ and appropriate $\mu \in \R$, the maximal solutions of the eikonal-type boundary-value problem 
\begin{equation} \label{eik}
H(p+Dv,y) = \mu \quad \mbox{in} \ \Rd \setminus \{ x \} \quad \mbox{and} \quad v(x)=0.
\end{equation}
In this section, we study properties of \eqref{eik} which are independent of the random environment, and so we drop the dependence on~$\omega$. We continue to assume the obvious $\omega$-independent analogues of \eqref{reg}, \eqref{coer} and \eqref{sqc}.

If $\Lip$ denotes the set of real-valued global Lipschitz functions on $\Rd$, we write $\SLp$ (resp. $\SLm$) for the subset of $\Lip$ consisting of the functions which are sublinear from below (resp. above)
\begin{equation*}
\SLp : = \Big\{ w \in \Lip \, : \, \liminf_{|y|\to \infty} |y|^{-1} w(y) \geq 0 \Big \} \quad \Big(\mbox{resp.} \quad \SLm := \Big\{ w \in \Lip  \, : \, \limsup_{|y|\to \infty} |y|^{-1} w(y) \leq 0 \Big \} \Big).
\end{equation*}
The set $\SL$ of Lipschitz functions which are strictly sublinear at infinity is $\SL := \SLp\cap \SLm$.

Next we define the quantities $\overline H_+(p)$ and $\overline H_-(p)$ by
\begin{equation}\label{Hbarpm}
\overline H_\pm(p) : = \inf\big\{ \nu \in \R \, : \, \mbox{there exists} \ w \in \SL^\pm \ \mbox{satisfying} \ H(p+Dw,y) \leq \nu \ \mbox{in} \ \Rd \big\},
\end{equation}
with the differential inequality in~\eqref{Hbarpm} interpreted in the viscosity sense. In light of Lemma~\ref{convtrick}, we may write
\begin{equation}\label{Hbarpm2}
\overline H_\pm(p) : = \inf_{w \in \SL^\pm} \esssup_{y\in \Rd} H(p+Dw,y).
\end{equation}
Also define
\begin{equation} \label{Hbarstar}
\overline H_* : =  \inf_{w \in \Lip} \esssup_{y\in \Rd} H(Dw,y) \qquad \mbox{and} \qquad \widehat H(p) : =  \inf_{w \in \SL} \esssup_{y\in \Rd} H(Dw,y). 
\end{equation}
According to the definitions and \eqref{reg}, for every $p\in \Rd$,
\begin{equation}\label{hoops}
\overline H_* \leq \min\big\{ \overline H_+(p) , \, \overline H_-(p) \big\} \leq \widehat H(p) \leq \sup_{y\in \Rd} H(p,y) < \infty.
\end{equation}
It is likewise immediate from the properties of $H$ and \eqref{Hbarpm2} that 
\begin{equation*}\label{}
p\mapsto \overline H_\pm(p) \quad \mbox{and} \quad p\mapsto \widehat H(p)  \quad \mbox{are continuous, coercive and satisfy~\eqref{sqc}.}
\end{equation*}
Indeed, for every $p,q\in \Rd$,
\begin{equation} \label{transconv}
\begin{aligned}
\overline H_\pm\left(\tfrac12 (p+q)\right) & = \inf_{v,w\in \SL^\pm} \esssup_{y\in \Rd} H\left( \tfrac12 (p+q) + \tfrac12 (Dv(y)+Dw(y)),y\right) \\
& \leq \inf_{v,w\in \SL^\pm} \esssup_{y\in \Rd} \Lambda \left( H(p+Dw(y),y),\, H(q+Dv(y),y) \right) \\
& \leq \inf_{v,w\in \SL^\pm} \Lambda \left(\esssup_{y\in \Rd}H(p+Dw(y),y), \, \esssup_{y\in \Rd} H(q+Dv(y),y) \right) \\
& = \Lambda \left(  \overline H_\pm(p), \, \overline H_\pm(q) \right),
\end{aligned}
\end{equation} 
where in each of the last two lines we used~\eqref{Lambda}. 

In our analysis we also need to consider maximal solutions of the eikonal problems 
\begin{equation*}\label{EIK-}
G(Dv,y) = \mu \quad \mbox{in} \ \Rd \setminus \{ x\} \quad \mbox{and} \quad v(x) = 0,
\end{equation*}
where $G(p,y,\omega):= H(-p,y,\omega)$. According to \eqref{Hbarpm2}, we have $\overline H_+(p) = \overline G_-(-p)$, $\widehat H(p) = \widehat G(-p)$ and $\overline H_* = \overline G_*$. The reader should therefore bear in mind that, although many of the intermediate results below are stated only for $\overline H_+$ or $\overline H_-$, analogous results also hold for the other (and we must be careful to keep track of minus signs). 

The following comparison principle was proved in \cite{ASo} for $H$ convex in $p$. The argument needs a modification to generalize to level-set convex $H$.

\begin{prop}\label{comp}
Assume $\mu > \widehat H(p)$ and $u,-v\in \USC(\Rd)$ satisfy
\begin{equation}\label{compde}
H(p+Du,y) \leq \mu \leq H(p+Dv,y) \quad \mbox{in} \ \Rd \setminus K,
\end{equation}
for a compact subset $K$ of $\Rd$, as well as the growth condition
\begin{equation}\label{gc-comp}
\liminf_{|y| \to \infty} |y|^{-1} v(y) \geq 0.
\end{equation}
Then
\begin{equation}\label{comp-conc}
\sup_{\Rd} (u-v) = \max_{K} (u-v).
\end{equation}
\end{prop}
\begin{proof}
With no loss of generality we assume that $p=0$. We first prove the result under the stronger hypothesis that
\begin{equation} \label{addhyp}
\liminf_{|y| \to \infty} |y|^{-1} v(y) > 0.
\end{equation}
Since $u$ is Lipschitz (Lemma~\ref{convtrick}), 
\begin{equation} \label{udn}
A:= \limsup_{|y| \to \infty} |y|^{-1} u(y) < \infty.
\end{equation}
Define
\begin{equation*}
E:= \Big\{ 0 \leq \lambda  \leq 1 \, : \, \liminf_{|y| \to \infty} \frac{v(y) - \lambda u(y)}{|y|} \geq 0 \Big\} \quad \mbox{and} \quad \alpha : = \sup E.
\end{equation*}
According to \eqref{addhyp} and \eqref{udn}, we have $\alpha > 0$ and $[0,\alpha) \subseteq E$. Observe that, for any $ 0< \ep < \alpha$,
\begin{equation*}
\liminf_{|y| \to \infty} \frac{v(y) - \alpha u(y)}{|y|} \geq \liminf_{|y| \to \infty} \frac{v(y) - (\alpha -\ep) u(y)}{|y|} - \ep \limsup_{|y| \to\infty} \frac{u(y)}{|y|} \geq -\ep A,
\end{equation*}
and hence, letting $\ep \to 0$, we get $\alpha \in E$ and $E=[0,\alpha]$. 

\smallskip

Next we show that $\alpha=1$ and hence $E=[0,1]$. Fix $0< \lambda <1 $ such that $\lambda \leq \alpha$ and select $\widehat H(0) < \nu < \mu$ and $w\in \SL$ satisfying 
\begin{equation*}\label{} 
H(Dw,y) \leq \nu \quad \mbox{in} \ \Rd.
\end{equation*}
Also fix $R,\ep > 0$ and $0 \leq \delta < \min\{ \tfrac12(1-\lambda), \ep/2A\}$ and define the auxiliary function
\begin{equation*}
\varphi_R(y):= (R^2+ |y|^2)^{\frac12} - R
\end{equation*}
as well as
\begin{equation*}
\widetilde u : = (\lambda + \delta)u + (1-\lambda-\delta) w \qquad \mbox{and} \qquad\widetilde v : = v + \ep \varphi_R.
\end{equation*}
It follows from \eqref{sqclambda} that
\begin{equation}\label{kabs}
H(D\widetilde u,y) \leq \widetilde \mu \quad \mbox{in} \ \Rd \setminus K
\end{equation}
where $\widetilde \mu:= \Lambda_{\lambda_1} (\widehat H(0),\nu) < \mu$ and $\lambda_1:=\min\{ \lambda,(1-\lambda)/2\}$. On the other hand, according to Lemma~\ref{yeswecan}, we may select $\ep > 0$ sufficiently small, depending on $\lambda$, that for some $\widetilde \mu < \widetilde \nu < \mu$,
\begin{equation} \label{vabs}
H(D\widetilde v,y) \geq \widetilde \nu  \quad \mbox{in} \ \Rd \setminus K.
\end{equation}
Using \eqref{udn}, we find that
\begin{equation*}
\liminf_{|y| \to \infty} \frac{\widetilde v(y)- \widetilde u(y)}{|y|} \geq \ep/2A > 0.
\end{equation*}
We may now apply the usual comparison principle (c.f.~\cite{CIL}) in bounded domains to conclude that
\begin{equation*}
\widetilde u - \widetilde v \leq \max_{K} (\widetilde u - \widetilde v)\quad \mbox{in} \ \Rd \setminus K. 
\end{equation*}
Letting $R\to \infty$, we deduce that $\widetilde u - v \leq \max_{K} (\widetilde u-v)$ in $\Rd \setminus K$, that is,
\begin{equation}\label{buk}
(\lambda + \delta)u + (1-\lambda-\delta) w -v \leq \max_{K} \left( (\lambda + \delta)u + (1-\lambda-\delta) w - v \right).
\end{equation}
Dividing by~$|y|$ and taking the limsup as $|y| \to \infty$ yields, using $w\in \SL$,
\begin{align*}\label{}
\limsup_{|y|\to\infty} |y|^{-1} \left( (\lambda+\delta)u(y) - v(y) \right) & = \limsup_{|y|\to\infty} |y|^{-1} \left( (\lambda+\delta)u(y) + (1-\lambda-\delta) w(y) - v(y) \right) \\
& \leq \limsup_{|y| \to \infty} |y|^{-1} \max_{z\in K} \left( (\lambda+\delta)u(z) + (1-\lambda-\delta) w(z) - v(z) \right) \\
& = 0.
\end{align*}
Hence $\lambda + \delta \in E$. If $\alpha < 1$, then we obtain a contradiction to the definition of $\alpha$ by choosing $\lambda = \alpha$ and $\delta > 0$. It follows that $\alpha=1$ and hence $E=[0,1]$. We may then take $\delta =0$ and send $\lambda \to 1$ in~\eqref{buk} to deduce that $u - v \leq \max_K (u-v)$ in $\Rd$. The proof of the Proposition under the additional hypothesis~\eqref{addhyp} is complete.

To remove the latter, let $\widetilde u$ and $\widetilde v$ be defined as above with $\delta =0$. Since $\liminf_{|y| \to \infty} |y|^{-1} \widetilde v(y) \geq \ep > 0$, we obtain from \eqref{kabs}, \eqref{vabs} and the argument above that
\begin{equation*}
\widetilde u - \widetilde v \leq \max_K (\widetilde u - \widetilde v) \quad \mbox{in} \ \Rd \setminus K.
\end{equation*}
Sending first $\ep \to 0$ and then $\lambda \to 1$ yields \eqref{comp-conc}. 
\end{proof}

We now construct maximal solutions of the eikonal equation which vanish at a particular point, and review some of their important properties. 

\begin{prop} \label{exist}
For each $\mu \geq \overline H_*$ and $x\in \Rd$, there exists a  solution $m_\mu(\,\cdot\,,x) \in \Lip$ of
\begin{equation}\label{metmx}
H(D_ym_\mu(\cdot,x),y) = \mu \quad \mbox{in} \ \Rd \setminus \{ x \}, \qquad m_\mu(x,x) = 0,
\end{equation}
which is \emph{maximal} in the sense that, if $w\in \Lip$ is a subsolution of $H(Dw,y) \leq \mu$ in $\Rd$, then $w(\cdot) - w(x) \leq m_\mu(\, \cdot \, , x)$ in $\Rd$, and \emph{subadditive}, i.e., for every $x,y,z\in \Rd$,
\begin{equation}\label{subaddeq}
m_{\mu}(z,x) \leq m_{\mu}(y,x) + m_{\mu}(z,y).
\end{equation}
Moreover, for every $\mu > \overline H_*$, 
\begin{equation}\label{charHp}
\liminf_{|y|\to \infty} \, |y|^{-1} \big( m_\mu(y,x) - p\cdot y \big) \geq 0 \quad \mbox{if and only if} \quad \mu \geq \overline H_+(p).
\end{equation}
Finally, if $\mu > \widehat H(p)$, then
\begin{equation}\label{mmup}
y\mapsto m_{\mu,p}(y,x):= m_\mu(y,x) - p\cdot (y-x)
\end{equation}
is the unique solution of \eqref{eik} belonging to $\SLp$.
\end{prop}
\begin{proof}
The existence of $m_\mu$ is a consequence of Perron's method and the fact that the coercivity of $H$ provides supersolutions to which we may apply the comparison principle of Proposition~\ref{comp}. It follows from \eqref{coer} that, for every $\mu \in \R$, there exists $C_\mu > 0$ sufficiently large so that, for every $x\in \Rd$, the function $\psi(y):= C_\mu |y-x|$ satisfies 
\begin{equation}\label{supertest}
H(D\psi, y) \geq \widehat \mu:=\max \{ \mu , \widehat H(0) \} + 1\quad\mbox{in} \ \Rd \setminus \{ x \}.
\end{equation}
Fix $\mu > \overline H_*$, $x\in \Rd$, define, for each $y\in \Rd$,
\begin{equation}\label{defmmu}
m_\mu(y,x) : = \sup \big\{ w(y) -w(x) \, : \, w\in \Lip \ \mbox{and} \ H(Dw,y) \leq \mu \ \mbox{in} \ \Rd \big\},
\end{equation}
and notice that, in view of Lemma~\ref{convtrick}, the differential inequality in \eqref{defmmu} may be interpreted either in the viscosity sense or in the almost everywhere sense. The set of admissible $w$ in \eqref{defmmu} is nonempty since $\mu > \overline H_*$. Since $\widehat \mu > \widehat H(p)$ and $\widehat \mu > \mu$, Proposition~\ref{comp}, applied to each of the admissible $w$ in \eqref{defmmu}, yields that 
\begin{equation}\label{uppbndme}
m_\mu(y,x) \leq C_{ \mu} |y-x|.
\end{equation}
In particular, $m_\mu(y,x)$ is finite. That $m_\mu(\cdot,x)$ is a subsolution of $H(Dm_\mu(\cdot,x),y)\leq \mu$ in $\Rd\setminus \{x\}$ is immediate, since it is the supremum of a collection of subsolutions. In fact, by \eqref{metmx} and Lemma~\ref{convtrick}, we see that $m_\mu(\cdot,x)$ is a \emph{global} subsolution, that is,
\begin{equation}\label{globalsubeq}
H(Dm_\mu(\cdot,x),y) \leq \mu \quad\mbox{in} \ \Rd.
\end{equation}
The fact that $m_\mu(\cdot,x)$ is a supersolution of \eqref{metmx} in the viscosity sense follows from the standard Perron argument adapted to viscosity solutions (c.f. \cite{Ba,CIL}).
The maximality of $m_\mu(\cdot,x)$ is obvious from its definition. The subadditivity property is immediate from the maximality of $m_\mu(\, \cdot \,,y)$ and the observation that for fixed $x,y\in \Rd$, $w(z) := m_\mu(z,x) - m_\mu(y,x)$ is a subsolution of 
\begin{equation*}
H(Dw,z) \leq \mu \quad \mbox{in} \ \Rd.
\end{equation*}
The characterization of $\overline H_+$ in  \eqref{charHp} is a consequence of \eqref{Hbarpm2} and the maximality of $m_\mu(\cdot,x)$, while the last claim follows from Proposition~\ref{comp} and \eqref{charHp}. This completes the construction for $\mu > \overline H_*$.

Observe that $m_\mu(\cdot,x)$ is Lipschitz with constant $C_\mu$. Indeed, owing to \eqref{subaddeq} and \eqref{uppbndme}, we have, for every $x,y,z\in \Rd$,
\begin{equation}\label{mmuLip}
m_\mu(z,x) - m_\mu(y,x) \leq m_\mu(z,y) \leq C_\mu|z-y|. 
\end{equation}
Moreover, it is immediate from \eqref{defmmu} that, for each $x,y\in\Rd$,
\begin{equation}\label{mmuinc}
\mu \mapsto m_\mu(y,x) \quad \mbox{is increasing.}
\end{equation}
Therefore, we may define
\begin{equation}\label{mmubot}
m_{\overline H_*}(y,x):= \inf_{\mu > \overline H_*} m_\mu(y,x) = \lim_{\mu \ssearrow \overline H_*} m_\mu(y,x).
\end{equation}
The infimum in \eqref{mmubot} is finite, which also yields that, as $\mu \ssearrow \overline H_*$, $m_\mu$ converges to $m_{\overline H_*}$ locally uniformly in $\Rd$. This implies that $m_{\overline H_*}(\, \cdot\,,x)$ is a solution of \eqref{metmx} with $\mu=\overline H_*$. The maximality of $m_{\overline{H}_*}$ is immediate from Proposition~\ref{comp} and \eqref{mmubot} and the subaddivity of $m_{\overline{H}_*}$ is immediate from that of $m_\mu$ with $\mu > \overline H_*$ and~\eqref{mmubot}. In light of \eqref{hoops}, the final two statements are vacuous in the case $\mu=\overline H_*$.
\end{proof}

We continue with several remarks that are used later in the paper.

\begin{remark}\label{mucontin}
The function $\mu \mapsto m_\mu(y,x)$ is continuous in $\mu$ in the sense that, if $\mu_j,\mu\geq \overline H_*$ are such that $\mu=\lim_{j\to \infty} \mu_j$, then, as $j\to \infty$ and locally uniformly in $\Rd\times\Rd$,
\begin{equation*}\label{}
m_{\mu_j} (y,x) \rightarrow m_\mu(y,x). 
\end{equation*}
This claim follows easily from an argument very similar to the final part of proof of Proposition~\ref{exist}, using Lipschitz estimates and the maximality property. 
\end{remark}

\begin{remark} \label{strcn}
It is useful to note that \eqref{mmuinc} can be improved. 
In fact, we claim that for every $k>1$, there exists $c>0$ such that, for all $\overline H_* \leq \nu < \mu \leq k$ with $\mu - \nu \geq k^{-1}$ and $x,y\in \Rd$,
\begin{equation}\label{str}
m_\mu(y,x) \geq m_\nu(y,x) + c|x-y|.
\end{equation}
This follows again from the maximality property and the observation that \eqref{reg} with $R=C_\mu+1$ yields that, for some constant $c> 0$ small enough, $y\mapsto m_\nu(y,x) + c |y-x|$ is a subsolution of $H(Du,y)\leq \mu$ in $\Rd$. 
\end{remark}

\begin{remark}\label{upsidedown}
Recall from the discussion preceding Proposition~\ref{comp} that $\overline H_* = \overline G_*$. Denote by $n_\mu$ the analogue of $m_\mu$ for $G$, that is,
\begin{equation} \label{defnmu}
n_\mu(y,x):= \sup \big\{ w(y) -w(x) \, : \, w\in \Lip \ \mbox{and} \ H(-Dw,y) \leq \mu \ \mbox{in} \ \Rd \big\}.
\end{equation}
By interchanging $w$ and $-w$ in \eqref{defmmu} and \eqref{defnmu} we see that, for all $x,y\in \Rd$,
\begin{equation}\label{updown}
n_\mu(y,x) = m_\mu(x,y).
\end{equation}
According to Lemma~\ref{convtrick}, $y\mapsto -n_\mu(y,x)=-m_\mu(x,y)$ satisfies
\begin{equation}\label{nmuglss}
H(D(-n_\mu(y,x)),y) \leq \mu \quad \mbox{in} \ \Rd.
\end{equation}
\end{remark}

\begin{remark} \label{dom}
It turns out that domination by $m_\mu$ is a necessary and sufficient condition for $w\in \Lip$ to be a subsolution, i.e., for each $w\in \Lip$ and $\mu \geq \overline H_*$,
\begin{equation} \label{domss}
w(y) - w(x) \leq m_\mu(y,x) \ \ \mbox{for every} \ x,y\in \Rd \quad \mbox{if and only if} \quad H(Dw,y) \leq \mu \ \mbox{in} \ \Rd.
\end{equation}
The necessity is proved in Proposition~\ref{exist}. To prove the other direction, we select a smooth test function $\phi$ and $x_0\in\Rd$ such that 
\begin{equation*}
y\mapsto (w-\phi)(y) \quad \mbox{has a local maximum at} \ y=x_0.
\end{equation*}
Assuming the first statement of \eqref{domss}, we deduce that
\begin{equation*}
y\mapsto w(x_0)- m_\mu(x_0,y) - \phi(y)  \quad  \mbox{has a local maximum at} \ y=x_0,
\end{equation*}
and, hence, in view of \eqref{updown},
\begin{equation*}
y\mapsto -n_\mu(y,x_0) - \phi(y) \quad \mbox{has a local maximum at} \ y=x_0.
\end{equation*}
In light of \eqref{nmuglss}, we obtain the desired conclusion that $H(D\phi(x_0) ,x_0) \leq \mu$.
\end{remark}

\section{Homogenization of the exterior eikonal problem}
\label{MP}

We continue the study of \eqref{eik}, now requiring $H$ to depend on the random variable $\omega$ with all of the hypotheses of Section~\ref{P} in force. Since each of the quantities studied in Section~\ref{EIK} is well-defined for every fixed $\omega\in \Omega$, we write
\begin{equation}\label{Hbarpm2-o}
\overline H_\pm(p,\omega) : = \inf_{w \in \SL^\pm} \esssup_{y\in \Rd} H(p+Dw(y),y,\omega),
\end{equation}
\begin{equation} \label{Hbarstar-o}
\overline H_*(\omega) : =  \inf_{w \in \Lip} \esssup_{y\in \Rd} H(Dw(y),y,\omega)  \qquad \mbox{and}\qquad \widehat H(p,\omega): =  \inf_{w \in \SL} \esssup_{y\in \Rd} H(Dw(y),y,\omega).
\end{equation}
Clearly $\omega \mapsto \overline H_*(\omega)$ and $(p,\omega) \mapsto \overline H_\pm(p,\omega),\ \widehat H(p,\omega)$ are measurable with respect to the appropriate $\sigma$-algebras. Moreover, they are invariant under the translation group $( \tau_y )_{y\in\Rd}$. Hence, in view of the ergodic hypothesis~\eqref{ergo}, there exists an event $\widehat\Omega \in \mathcal F$ of full probability such that, for every $\omega\in \widehat\Omega$ and $p\in \Rd$,
\begin{equation}\label{Hbarcon}
\overline H_\pm(p,\omega) = \overline H_\pm(p), \qquad \overline H_*(\omega) = \overline H_* \qquad \mbox{and} \qquad \widehat H(p,\omega) = \widehat H(p).
\end{equation}
Adapting \eqref{defmmu}, we define, for every $\mu \geq \overline H_*$,
\begin{equation}\label{defmmu-o}
m_\mu(y,x,\omega) : = \sup \big\{ w(y) -w(x) \, : \, w\in \Lip \ \mbox{and} \ H(Dw,y,\omega) \leq \mu \ \mbox{in} \ \Rd \big\}.
\end{equation}
It is immediate that the $m_\mu$'s are measurable with respect to the appropriate $\sigma$-algebras and possess, for each fixed $\omega\in \Omega$, the properties in Proposition~\ref{exist}. Finally, the stationarity of $H$ yields that, for all $\mu \geq \overline H_*$ and $x,y,z\in\Rd$,
\begin{equation} \label{dfmeas-o}
 m_\mu(y,x,\tau_z\omega) = m_\mu(y+z,x+z,\omega).
\end{equation}

Next we utilize the subadditive ergodic theorem (Proposition~\ref{SAET}) and the subadditivity of $m_\mu$ to essentially homogenize the eikonal equation. The argument is the same as the one given in \cite{ASo,ASo2,DS1}.

\begin{prop} \label{limitfuns}
There exists an event $\Omega_1\in\mathcal F$ of full probability and $\overline m:\left[\,\overline H_*,\infty\right)\times\Rd \to \R$, $\overline m = \overline m_\mu(y)$, such that, for every $\omega\in \Omega_1$, $x,y\in \Rd$, and $\mu \geq \overline H_*$,
\begin{equation}\label{limiteq}
\lim_{t\to \infty} t^{-1}m_\mu(ty,tx,\omega) = \overline m_\mu(y-x).
\end{equation}
For fixed $\mu \geq \overline H_*$, $\overline m_\mu(\cdot) \in \Lip$ is positively homogeneous and convex and, for $y\neq 0$, $\mu \mapsto \overline m_\mu(y)$ is strictly increasing on $\big(\overline H_*,\infty\big)$. 
\end{prop}
\begin{proof}
The existence of $\overline m_\mu \in \Lip$ and the limit \eqref{limiteq} are obtained, in the case $x=0$, from an application of the subadditive ergodic theorem. We must simply check that the hypotheses of this theorem are satisfied, with \eqref{subaddeq} in mind. This is easy, and since the argument is nearly the same as in \cite{ASo,ASo2,DS1}, we omit the proof. The result for general $x$ then follows from the Lipschitz estimates for $m_\mu$ and a combination of the ergodic theorem and Egoroff's theorem. 

It is clear from the form of the limit \eqref{limiteq} that $\overline m_\mu$ is positively homogeneous. Observe also that \eqref{subaddeq} implies that, for every $y,z\in \Rd$, 
\begin{equation*}
\overline m_\mu(y) = \lim_{t\to\infty} t^{-1} m_\mu(ty,0,\omega) \leq \lim_{t\to\infty} t^{-1} \big( m_\mu(tz,0,\omega) +  m_\mu(ty,tz,\omega) \big)
= \overline m_\mu(z) + \overline m_\mu(y-z).
\end{equation*}
Therefore $m_\mu$ is convex. That the map $\mu \mapsto \overline m_\mu(y)$ is strictly increasing for $y\neq 0$ is a consequence of \eqref{str}. 
\end{proof}

Next we study the relationship between $\overline m_\mu$ and $\overline H_{\pm}$. The fact is that their $\mu$-sublevel sets are dual to each other as convex sets. We have:

\begin{lem}
For every $\mu > \overline H_*$,
\begin{equation}\label{formmbareq2}
\overline m_\mu(y) = \max\big\{ p\cdot y \, : \, \overline H_+(p) \leq \mu \big\}.
\end{equation}
\end{lem}
\begin{proof}
Since $\overline m_\mu$ is positively homogeneous and convex it can be written as
\begin{equation}\label{formmbareq}
\overline m_\mu(y) = \max\{ p\cdot y \, : \, p \in K_\mu \} \qquad \mbox{where} \qquad K_\mu := \bigcup_{z\in \Rd} \partial \overline m_\mu(z) = \bigcup_{|z| = 1} \partial \overline m_\mu(z),
\end{equation}
where $\partial f(z)$ denotes the usual subdifferential of a convex function $f$ at a point $z$.

The set $K_\mu$ is nonempty, closed and convex. Observe that, by \eqref{formmbareq}, \eqref{limiteq},  \eqref{charHp} and the positive homogeneity of $\overline m_\mu$, we have the following chain of equivalences:
\begin{align} \nonumber
p\in K_\mu  & \quad \mbox{if and only if} \quad \overline m_\mu(y) \geq p\cdot y \quad \mbox{for every} \ y\in\Rd \\ \label{obvious}
& \quad \mbox{if and only if} \quad \liminf_{|y|\to \infty} |y|^{-1} \big( m_\mu(y,0,\omega) - p\cdot y\big) \geq 0 \quad \mbox{for every} \ \omega\in \Omega_1 \\ & \quad \mbox{if and only if} \quad \mu \geq H_+(p). \nonumber
\end{align}
Hence $K_\mu = \{ p \, : \, \overline H_+(p) \leq \mu \}$.
\end{proof}

\begin{cor} \label{charHstr}
We have
\begin{equation}\label{minchar}
\overline H_* = \min_{p\in \Rd} \overline H_+(p).
\end{equation}
\end{cor}
\begin{proof}
According to Proposition~\ref{limitfuns}, $\overline m_{\overline H_*}(\cdot)$ is convex and $\overline m_{\overline H_*} (0)=0$. Therefore, there exists $p\in \partial m_{\overline H_*}(0)$, that is, $p\cdot y \leq \overline m_{\overline H_*}(y)$ for all $y\in \Rd$. Since $\mu \mapsto m_\mu(y)$ is increasing, it follows that $p\cdot y \leq \overline m_\mu(y)$ for all $y\in \Rd$ and $\mu > \overline H_*$. According to~\eqref{obvious}, we deduce that $\overline H_+(p) \leq \overline H_*$. The claim now follows using~\eqref{hoops}.
\end{proof}

\begin{cor}
For all $p\in \Rd$,
\begin{equation}\label{trip}
\overline H_+(p) = \overline H_-(p)
\end{equation}
\end{cor}
\begin{proof}
In light of Remark~\ref{upsidedown}, \eqref{minchar}, and the fact that $\overline H_-(p) = \overline G_+(-p)$ and $\overline H_*=\overline G_*$, we have $\overline H_* = \min_{\Rd} \overline H_-$. With the notation of Remark~\ref{upsidedown}, it suffices, in view of  \eqref{formmbareq2}, to show that $\overline n_\mu(y) = \overline m_\mu(-y)$ for every $\mu > \overline H_*$ and $y\in \Rd$, a fact which is immediate from \eqref{updown} and \eqref{limiteq}.
\end{proof}

We henceforth define
\begin{equation}\label{allequal}
\overline H(p) : = \overline H_+(p) = \overline H_-(p).
\end{equation}
Observe that \eqref{transconv} implies that, for every $\mu > \overline H_*$, the $\mu$-level set of $\overline H$ has empty interior. In particular, for every $p\in \Rd$,
\begin{equation}\label{pbndry}
\overline H(p) > \overline H_*  \quad \mbox{implies that} \quad p\in \partial K_{\overline H(p)},
\end{equation}
where $\partial E$ denotes the usual Euclidean boundary of a set $E \subseteq \Rd$.

\begin{cor}\label{mmueqdC}
For each $\mu > \overline H_*$, $\overline m_\mu$ is the unique maximal solution of
\begin{equation}\label{mmueqd}
\overline H(D\overline m_\mu) = \mu \quad \mbox{in} \ \Rd \setminus \{ 0 \}, \quad \overline m_\mu(0) = 0.
\end{equation}
\end{cor}
\begin{proof}
It is clear from \eqref{formmbareq2} that \eqref{mmueqd} holds  a.e. in $\Rd$. To see that it holds in the viscosity sense as well, we notice first (from Lemma~\ref{convtrick}) that $\overline m_\mu$ is a subsolution of \eqref{formmbareq2} in $\Rd$. That $\overline m_\mu$ is a viscosity supersolution away from the origin follows from the fact that, if it can be touched from below by a smooth function $\varphi$ at a point $x_0 \neq 0$, then according to \eqref{pbndry}, $D\varphi(x_0) \in \partial K_\mu$ and thus, by \eqref{formmbareq} and the continuity of $\overline H$, we obtain $\overline H(D\varphi(x_0)) = \mu$, as desired. The maximality of $\overline m_\mu$ is obvious from~\eqref{formmbareq2}.
\end{proof}
%
%
%
%

\section{The proof of homogenization from the metric problem} \label{HP}

We present a direct argument to show that \eqref{limiteq} and \eqref{mmueqd} imply the homogenization of~\eqref{HJ}. We begin with the observation that Corollary~\ref{mmueqdC} can be seen (after rescaling) as a homogenization result for the metric problem. Indeed, setting
\begin{equation*}\label{}
m_\mu^\ep(x,z,\omega):= \ep m_\mu\left( \frac x\ep , \frac z\ep, \omega \right),
\end{equation*}
we see that
\begin{equation}\label{HJMrs}
H\left( Dm_\mu^\ep, \frac x\ep, \omega\right) = \mu \quad \mbox{in} \ \Rd \setminus \{ z \}, \quad m_\mu^\ep(z,z,\omega) = 0. 
\end{equation}
In other words, $m_\mu^\ep$ corresponds to the solutions given in Proposition~\ref{exist} for the rescaled Hamiltonian $H^\ep(p,y,\omega) = H(p,y/\ep,\omega)$. Proposition~\ref{limitfuns} and Corollary~\ref{mmueqdC} yield the assertion that, as $\ep \to 0$, the problem \eqref{HJMrs} homogenizes to \eqref{mmueqd}. Of course, this assertion is \emph{a priori} weaker than Theorem~\ref{H}. However, using a variation of the perturbed test function method of Evans~\cite{E}, we show that it is actually sufficient.

Before giving the proof of Theorem~\ref{H}, we recall the following well-known reduction. To demonstrate the homogenization of~\eqref{HJ}, it suffices to consider, for $\delta > 0$, the \emph{auxiliary macroscopic problem}
\begin{equation} \label{mac}
\delta v^\delta + H(p+Dv^\delta,y,\omega) = 0 \quad \mbox{in} \ \Rd
\end{equation}
and to show that the unique bounded, uniformly continuous solution $v^\delta=v^\delta(y,\omega\,;p)$ of \eqref{mac} satisfies
\begin{equation}\label{convmac0}
\overline H(p)=-\lim_{\delta \to 0} \delta v^\delta(0,\omega\,;p) \quad \mbox{a.s. in} \ \omega.
\end{equation}
We note that a direct comparison with constant functions and \eqref{coer} yield (see~\cite{ASo} for details) a positive constant $C> 0$ such that, for all $x,y\in\Rd$,
\begin{equation} \label{delvdellip}
\big| \delta v^\delta(x,\omega\,;p)\big| \leq C \quad \mbox{and} \quad \big| v^\delta(y,\omega\,;p) - v^\delta(z,\omega\,;p) \big| \leq C|z-y|.
\end{equation}

For $p\in \Rd$, define
\begin{equation*}
h_* (p,\omega): = \liminf_{\delta \to 0} \, -\delta v^\delta(0,\omega\,;p) \quad \mbox{and} \quad h^* (p,\omega): = \limsup_{\delta \to 0} \, -\delta v^\delta(0,\omega\,;p).
\end{equation*}
It is easy to check, using \eqref{delvdellip}, that $h_*(p,\cdot)$ and $h^*(p,\cdot)$ are invariant under the translation group $( \tau_y )_{y\in \Rd}$. Hence both are constant almost surely by \eqref{ergo}, and we may write $h_*(p,\omega) = h^*(p)$ and $h_*(p,\omega) = h_*(p)$ on set of full probability.

An easy comparison argument using \eqref{reg} yields that
\begin{equation} \label{fmbsdf}
\big\{ \delta v^\delta(y,\omega\,;\cdot) \, : \, \delta > 0, \ (y,\omega)\in \Rd\times \Omega \big\} \quad \mbox{is bounded and equicontinuous on} \ \Rd.
\end{equation}
Indeed, to get a uniform modulus of continuity for $\delta v^\delta$ in the variable $p$, we simply add an appropriate constant to $v^\delta(\cdot,\omega; q)$, insert it into \eqref{mac}, and compare to $v^\delta(\cdot,\omega;p)$ (see~\cite{ASo}). This permits us to prove \eqref{convmac0} for each fixed $p\in \Rd$. We then deduce that the event in which \eqref{convmac0} occurs for every rational $p$ has full probability, and hence, by \eqref{fmbsdf}, so also does the event in which \eqref{convmac0} occurs for every $p\in \Rd$.

The following lemma follows from a well-known argument which combines the Lipschitz estimate \eqref{delvdellip}, Egoroff's theorem and the ergodic theorem. The proof is nearly identical to that of (7.3) in~\cite{ASo}, but we include it for the reader's convenience.

\begin{lem} \label{ball1d}
There exists an event $\Omega_2\in \mathcal F$ of full probability such that,  for every $\omega\in \Omega_2$, $p\in \Rd$ and $R > 0$,
\begin{equation}\label{ball1de}
h_*(p) = \liminf_{\delta \to 0} \inf_{y\in B_{R/\delta}} -\delta v^\delta(y,\omega\,;p) \quad \mbox{and} \quad h^*(p) = \limsup_{\delta \to 0} \sup_{y\in B_{R/\delta}} -\delta v^\delta(y,\omega\,;p).
\end{equation}
\end{lem}
\begin{proof}
Since the arguments are nearly the same, we only prove the first identity of \eqref{ball1de}. For ease of notation, we drop the dependence on $p$. Fix $0 < \alpha <\frac12$. From the definition of $h_*(p)$ and Egoroff's theorem, there exists $\delta_\alpha> 0$ and an event $E_\alpha \subseteq \Omega$ with probability $\Prob[E_\alpha] \geq 1-\alpha$ such that, for every $0 < \delta \leq  \delta_\alpha$,
\begin{equation*}
\inf_{\omega\in E_\alpha} -\delta v^\delta(0,\omega) \geq h_* -\alpha.
\end{equation*}
According to the ergodic theorem (c.f. the multiparameter version proved in Becker~\cite{B}), there exists a subset $F_\alpha \subseteq \Omega$ of full probability such that, for every $\omega \in F_\alpha$,
\begin{equation*}
\lim_{R \to \infty} \fint_{B_R} \mathds{1}_{E_\alpha}(\tau_y\omega) \, dy = \Prob[E_\alpha] \geq 1-\alpha.
\end{equation*}
Define $F_0 : = \cap_{j=1}^\infty F_{2^{-j}}$ so that $\Prob[F_0] =1$. Fix $\omega \in F_0$ such that $h_*(\omega) = h_*$ and $R, \rho > 0$, with $\rho = 2^{-j}$ for some $j\in \N$. It follows that, if $\delta > 0$ is sufficiently small (depending on $\omega$, $R$ and $\rho$), then
\begin{equation} \label{fillingup}
| \{ y\in B_{R/\delta} : \tau_y \omega \in E_\rho \} | \geq (1-2\rho) |B_{R/\delta}|.
\end{equation}
Select any $z\in B_{R/\delta}$. According to~\eqref{fillingup}, there exists a point $y\in B_{R/\delta}$ with $|y-z|\leq C\rho R\delta^{-1}$ and $\tau_y\omega \in E_\rho$. In view of the Lipschitz estimate \eqref{delvdellip} and the stationarity of the $v^\delta$'s, we deduce that, for each $\delta$ sufficiently small, depending on $\omega$, $\rho$, and $R$,
\begin{align*}
-\delta v^\delta(z,\omega) - h_* & \geq -| \delta v^\delta(z,\omega) - \delta v^\delta(y,\omega)| - \delta v^\delta(y,\omega) - h_*
\\ & \geq -C\rho R - \delta v^\delta(0,\tau_y\omega) - h_* \\ & \geq -C\rho R - \rho.
\end{align*}
Since $\rho$ can be made arbitrarily small, we deduce that, for each $\omega \in F_0\cap \{ \omega\,:\, h_*(p,\omega)=h_*(p)\}$ and $R>0$,
\begin{equation*}
\liminf_{\delta\to 0} \inf_{z\in B_{R/\delta}} -\delta v^\delta(z,\omega)  \geq h_*,
\end{equation*}
which is the desired conclusion.
\end{proof}

We now give the proof of the homogenization of \eqref{HJ}.

\begin{proof}[{Proof of Theorem~\ref{H}}]
We argue that $ h_*(p)=h^*(p)=\overline H(p)$. With $\Omega_1$ and $\Omega_2$ as in Proposition~\ref{limitfuns} and Lemma~\ref{ball1d}, respectively, fix $\omega\in \Omega_1\cap \Omega_2:=\Omega_0$ and $p\in \Rd$, define
\begin{equation}\label{wdel}
w^\delta(y): = v^\delta(y,\omega\,;p) - v^\delta(0,\omega\,;p) + p\cdot y,
\end{equation}
and observe that $w^\delta$ satisfies the equation
\begin{equation}\label{wdeleq}
H(Dw^\delta,y,\omega) = -\delta v^\delta(y,\omega\,;p) \quad \mbox{in} \ \Rd.
\end{equation}
Owing to \eqref{delvdellip}, we can find a subsequence $\delta_j\to 0$, which depends on $\omega$, and a function $w \in \Lip$ such that, as $j\to \infty$, 
$w^{\delta_j} \to w$ locally uniformly in $\Rd$ and
\begin{equation} \label{spcdel}
h_*(p)  = - \lim_{j \to \infty}  \delta_j v^{\delta_j}(0,\omega\,;p). \end{equation}
By passing to limits and using \eqref{delvdellip}, we deduce that $w$ is a solution of
\begin{equation*}\label{}
H(Dw,y,\omega) = h_*(p) \quad \mbox{in} \ \Rd, 
\end{equation*}
and, hence, $\overline H_* \leq h_*(p)$.

We next show by a comparison argument similar to one presented in the proof of~\cite[Proposition~7.1]{ASo} that $\overline H(p) \leq h_*(p)$. Suppose on the contrary that $\alpha: = \overline H(p)- h_*(p) > 0$. For convenience, denote $\mu:= \overline H(p)$ and observe that $\mu - \overline H_*\geq \alpha > 0$. It follows from \eqref{pbndry} that $p\in \partial K_\mu$, and thus, by \eqref{formmbareq}, there exists $x\in \Rd$ with $|x| =1$ such that $p \in \partial \overline m_\mu(x)$. In other words,
\begin{equation}\label{keyx}
0 = \overline m_\mu(x) - p\cdot x \leq 0 \leq \overline m_\mu(y) - p\cdot y \quad \mbox{for every} \ y\in \Rd.
\end{equation}
Our argument relies on the fact that, with this choice of $x$, the function $y\mapsto m_\mu(y,-tx,\omega) - p\cdot y$, for large enough $t> 0$, acts as an ``approximate supercorrector" in the ball of radius $t/2$ centered at the origin, by which we mean that it is nearly a supersolution of the macroscopic problem and is appropriately bounded from below.

Fix $0<r\leq1$ small enough that $Cr <  \alpha/8$, where $C> 0$ is the Lipschitz constant of the $v^\delta$'s in~\eqref{delvdellip}, and find a subsequence $\delta_j\to 0$ such that~\eqref{spcdel} holds. Recall that the subsequence depends on $\omega$ but the latter is chosen from the event $\Omega_1\cap \Omega_2$, which is of full probability. We have, for large enough~$j$, 
\begin{equation*}\label{}
H(Dw^{\delta_j},y,\omega) \leq h_*(p) + \frac14\alpha \quad \mbox{in} \ B_{r/\delta_j}.
\end{equation*}
Define the functions 
\begin{equation} \label{hats}
\widehat m_\mu(y,z) : = m_\mu(y,z,\omega) - m_\mu(0,z,\omega) \quad \mbox{and} \quad \widehat m_\mu^\delta(y) : = \widehat m_\mu(y,-x/\delta) + \ep (r^2+3|y|^2)^{\frac12} - \ep r.
\end{equation}
By taking $\ep > 0$ small enough, depending on $\alpha$, we have
\begin{equation*}\label{}
H( D\widehat m_\mu^\delta,y,\omega) \geq \mu - \frac14 \alpha = h_*(p) + \frac34 \alpha \quad \mbox{in} \ \Rd \setminus \{ -x/\delta\} \supseteq B_{r/\delta}.
\end{equation*}
Since the comparison principle yields, for sufficiently large $j$ and small $\ep> 0$,
\begin{equation}\label{contr}
0 = w^{\delta_j}(0) - \widehat m_\mu^{\delta_j}(0) \leq \max_{y\in \partial B_{r/\delta_j}}  \left( w^{\delta_j}(y) - \widehat m_\mu^{\delta_j}(y) \right),
\end{equation}
we obtain a contradiction by showing that, for large enough $j$, the last term on the right is actually negative. Notice that, by \eqref{ball1de} and $\omega\in \Omega_2$, we have
\begin{equation}\label{delvdelbcs}
\limsup_{\delta \to 0} \sup_{y\in \partial B_{1/\delta}} \delta v^\delta(y,\omega\,;p) + h_*(p) = 0,
\end{equation}
and by $\omega\in \Omega_1$, \eqref{limiteq} and \eqref{keyx},
\begin{equation*}\label{mmudel}
\begin{aligned}
\lim_{\delta \to 0} \inf_{y\in \partial B(0,r/\delta)} \delta \big( \widehat m_\mu^\delta(y) -p\cdot y\big) & = \ep r + \lim_{\delta\to0} \inf_{z\in \partial B(0,r)} \delta \Big( m_\mu\Big(\frac z\delta , -\frac x\delta,\omega\Big) - m_\mu\Big(0,-\frac x\delta,\omega\Big) -  p\cdot\frac z \delta \Big) \\
& = \ep r + \inf_{z\in \partial B(0,r)} \big( \overline m_\mu(z+x) - \overline m_\mu(x) - p\cdot z \big) \\
& \geq \ep r.
\end{aligned}
\end{equation*}
Subtracting this inequality from \eqref{delvdelbcs}, using \eqref{spcdel}, we deduce that the right side of \eqref{contr} is indeed negative for all sufficiently large $j$. This completes the proof that $\overline H(p)\leq h_*(p)$. 

We have left to show that $h^*(p) \leq \overline H(p)$. We use an argument similar to the one above, which, however, is simpler in this case, since we need an ``approximate subcorrector" instead of an approximate supercorrector, and subsolutions are easier to obtain in this context than supersolutions. The subcorrector is constructed by essentially substituting  $y\mapsto -n_\mu(y,x,\omega) =- m_\mu(x,y,\omega)$ for $y\mapsto m_\mu(y,x,\omega)$ in the argument above, and changing some signs. For completeness, we give the details.

Arguing by contradiction, we suppose that $h^*(p) > \overline H(p)$. Select $\overline H(p) < \mu < h^*(p)$ and set $\alpha:= h^*(p) - \mu> 0$. Observe that $p\in\partial \overline m_\mu(0)$, which is equivalent to $-p \in \partial \overline n_\mu(0)$, which is equivalent to
\begin{equation}\label{keyx-n}
0\leq \overline n_\mu(y) + p\cdot y \quad \mbox{for every} \ y\in \Rd.
\end{equation}
Continuing to mimic the previous setup, let $0<r\leq1$ be as above and select a sequence $\delta_j  \to 0$, once again depending on $\omega$, such that
\begin{equation} \label{epdel-n}
\lim_{j\to \infty} -\delta_j v^{\delta_j}(\cdot,\omega\,;p) = h^*(p),
\end{equation}
and deduce that, for all large enough $j$,
\begin{equation*}
H( Dw^{\delta_j},y,\omega) \geq h^*(p) - \frac14\alpha \quad \mbox{in} \ B_{r/\delta_j}.
\end{equation*}
Denote
\begin{equation*}
\widehat n_\mu (y) :=  - n_\mu(y,0,\omega) - \ep |y|
\end{equation*}
and notice that, compared to \eqref{hats}, we have introduced a sign change, set $x=0$ and simplified the perturbative term, which does not need to be smooth since we may use Lemma~\ref{convtrick}. Owing to \eqref{nmuglss} and \eqref{reg}, we have, for $\ep > 0$ small enough, depending on $\alpha$,
\begin{equation*}
H\big(D\widehat n_\mu,y,\omega\big) \leq \mu + \frac14\alpha = h^*(p) - \frac34 \alpha \quad \mbox{in} \ \Rd.
\end{equation*}
The comparison principle yields, for large enough $j$,
\begin{equation} \label{contr-n}
0 = w^{\delta_j}(0) - \widehat n_\mu(0) \geq \min_{y\in \partial B(0,r/\delta_j)} \big( w^{\delta_j}(y) - \widehat n_\mu(y) \big).
\end{equation}
To contradict \eqref{contr-n}, we again use $\omega\in \Omega_2$ and \eqref{ball1de} to get
\begin{equation} \label{delvdelbcs-n}
\liminf_{\delta \to 0} \inf_{y\in B_{1/\delta}} \delta v^\delta(y,\omega\,;p) + h^*(p)  =0,
\end{equation}
and then apply $\omega\in \Omega_1$, \eqref{limiteq} and \eqref{keyx-n} to get 
\begin{equation}\label{mmudel-n}
\begin{aligned}
\lim_{\delta \to 0} \sup_{y\in \partial B(0,r/\delta)} \delta \big( \widehat n_\mu(y) -p\cdot y\big) & = -\ep r + \lim_{\delta\to0} \sup_{z\in \partial B(0,r)} \delta \Big( - n_\mu\Big(\frac z\delta,0,\omega \Big) -  p\cdot\frac z \delta \Big) \\
& = -\ep r + \sup_{z\in \partial B(0,r)} \big( -\overline n_\mu(z) - p\cdot z \big) \\ &\leq -\ep r.
\end{aligned}
\end{equation}
Subtracting \eqref{mmudel-n} from \eqref{delvdelbcs-n} and using \eqref{epdel-n}, we conclude that \eqref{contr-n} is violated for sufficiently large $j$. The proof is complete.
\end{proof}

\section{Qualitative properties of the effective Hamiltonian} \label{C}

We make several observations regarding the qualitative behavior of $\overline H$. Previous proofs relied on the existence of correctors (in the periodic setting) or subcorrectors which are strictly sublinear at infinity. Here we deduce it instead from the symmetry found in Remark~\ref{upsidedown}, specifically from \emph{one-sided} inf-sup formulas implied by \eqref{Hbarpm2} and \eqref{trip}, which as far as we know are know, i.e., a.s. in $\omega$,
\begin{equation}\label{1sinfsup}
\overline H(p) = \inf_{w\in\SLp} \esssup_{y\in\Rd} H(p+Dw,y,\omega) = \inf_{w\in\SLm} \esssup_{y\in\Rd} H(p+Dw,y,\omega).
\end{equation}

\begin{cor} \label{evenH}
If $p\mapsto H(p,y,\omega)$ is even, then $p\mapsto \overline H(p)$ is even. 
\end{cor}
\begin{proof}
According to~\eqref{1sinfsup}, a.s. in $\omega$,
\begin{align*}
\overline H(p) = \inf_{w\in\SLp} \esssup_{y\in\Rd} H(p+Dw(y),y,\omega) & = \inf_{w\in\SLp} \esssup_{y\in\Rd} H(-p-Dw(y),y,\omega) \\ & = \inf_{w\in\SLm} \esssup_{y\in\Rd} H(-p+Dw(y),y,\omega) = \overline H(-p).
 \qedhere
 \end{align*}
\end{proof}

It is well-known that the effective Hamiltonian may fail to be strictly convex even if $H$ is strictly convex in $p$. For example, in very general situations, $\overline H$ possesses a ``flat spot" (i.e., is constant on a set of nonempty interior) at its minimum, and its level sets may have flat edges (in the periodic setting, see Concordel~\cite{C} and E~\cite{WE}). However, the failure of strict convexity is restricted to the shape of the level sets. Indeed, recall from \eqref{transconv} that
\begin{equation}\label{Hbstrc}
\overline H\big( \tfrac12 (p+q) \big) \leq \Lambda\big(\overline H(p),\overline H(q) \big).
\end{equation}
It follows that $\overline H$ is strictly convex in directions transverse to its level sets, provided that $H$ is so, since $H$ is strictly convex in directions transverse to its level sets if and only if there exists $\Lambda$ as in \eqref{Lambda} such that \eqref{sqc} holds, and, for all $\mu,\nu \in \R$ with $\mu\neq \nu$, $\Lambda(\mu,\nu) < \tfrac12(\mu+\nu)$. This observation  simplifies as well as generalizes and precisely quantifies a result of Evans and Gomes~\cite{EG}, who proved a weaker version of it, using the methods of weak KAM theory, at points of differentiability of $\overline H$, in the periodic setting and under the assumption that $H$ is uniformly convex in $p$.

We conclude with a proof of existence, for any $p\in \Rd$, of exact subcorrectors, using an argument due to Lions and Souganidis~\cite{LS3}. In the level-set convex case, our arguments are \emph{a posteriori} in the sense that they rely on the homogenization result itself, namely \eqref{convmac0}. As a consequence we recover the usual form of the inf-sup formula for $\overline H(p)$, that is, we obtain that $\widehat H(p) = \overline H(p)$.

\begin{prop} \label{usual}
For every $p\in \Rd$, 
\begin{equation}\label{usualeq}
\overline H(p) = \inf_{w\in \SL} \esssup_{y\in \Rd} H(p+Dw(y),y,\omega) \quad \mbox{a.s. in} \ \omega.
\end{equation}
Moreover, there exists $w=w(y,\omega)$ with stationary gradient satisfying 
\begin{equation}\label{subcorrect}
H(p+Dw,y,\omega) \leq \overline H(p) \quad \mbox{in} \ \Rd \quad \mbox{and} \quad w(\cdot,\omega) \in \SL \quad \mbox{a.s. in} \ \omega.
\end{equation}
\end{prop}
\begin{proof}
One inequality in \eqref{usualeq} is given by \eqref{Hbarpm2-o} and \eqref{allequal}, and the other follows from the subcorrector $w$ we construct below.

Fix $p\in \Rd$ and define $w^\delta(y,\omega): = v^\delta(y,\omega\,;p) - v^\delta(0,\omega\,;p)$. According to the first inequality of~\eqref{delvdellip}, we may select a subsequence $\delta_j \to 0$ and a function $w\in L^\infty(\Rd \times \Omega)$ such that, as $j\to \infty$,
\begin{equation}\label{weakstar}
w^{\delta_j} \rightharpoonup w \quad\mbox{weakly-star in} \ L^\infty(\Rd\times \Omega).
\end{equation}
Since, by \eqref{delvdellip}, $w^{\delta_j}$ is uniformly Lipschitz continuous a.s. in $\omega$, it follows that $w$ is Lipschitz on $\Rd$ a.s. in $\omega$ with the same constant. We claim that $w$ satisfies \eqref{subcorrect}.

To demonstrate the differential inequality in \eqref{subcorrect}, it suffices by Remark~\ref{dom} to prove that, for every $x,y\in\Rd$ and a.s. in $\omega$,
\begin{equation}\label{clm}
w(y,\omega) - w(x,\omega) + p\cdot (y-x) \leq m_{\overline H(p)}(y,x,\omega).
\end{equation}
Fix $\mu > \overline H(p)$, $x\in \Rd$ and $\omega \in \Omega_1\cap \Omega_2$ so that the limits in Proposition~\ref{limitfuns} and Lemma~\ref{ball1d} hold. According to Remark~\ref{strcn}, \eqref{charHp}, \eqref{limiteq} and~\eqref{formmbareq2},
\begin{equation}\label{dfgh}
\liminf_{|y| \to \infty} |y|^{-1} \left( m_{\mu}(y,x,\omega) -p\cdot y\right) \geq c + \liminf_{|y| \to \infty} |y|^{-1} \left( m_{\overline H(p)}(y,x,\omega) -p\cdot y\right) \geq c.
\end{equation}
Using \eqref{delvdellip} and \eqref{dfgh}, there exists $R> 0$, sufficiently large and depending also on $x$, such that for  every large $j\in \N$ and $y\in \Rd \setminus B_{R/\delta_j}$,
\begin{equation}\label{strcneq}
 w^{\delta_j}(y,\omega) -w^{\delta_j}(x,\omega)+ p\cdot(y-x) \leq \frac{C}{\delta_j} + p\cdot (y-x) \leq \frac12 c|y-x| + p\cdot (y-x) \leq m_\mu(y,x,\omega).
\end{equation}
By \eqref{convmac0} and Lemma~\ref{ball1d}, for all sufficiently large $j\in \N$, 
\begin{equation*}\label{}
H(p+Dw^{\delta_j},y,\omega) = \delta_j w^{\delta_j}(y,\omega) \leq \mu \quad \mbox{in} \ B_{R/\delta_j}.
\end{equation*}
If follows that $\psi(y):= \max\left\{ w^{\delta_j}(y,\omega) -w^{\delta_j}(x,\omega)+ p\cdot(y-x), \, m_\mu(y,x,\omega) \right\}$ satisfies  $H(D\psi,y) \leq \mu$ in $\Rd$. By the maximality of $m_\mu(\cdot,x,\omega)$, we obtain that $\psi(y) \leq m_\mu(y,x,\omega)$ for all large $j\in\N$ and every $y\in \Rd$. Hence, for every $x,y\in \Rd$, $\omega\in \Omega_2$ and sufficiently large $j \in \N$,
 \begin{equation*}
 w^{\delta_j}(y,\omega) -w^{\delta_j}(x,\omega)+ p\cdot(y-x) \leq m_\mu(y,x,\omega).
 \end{equation*}
In other words, for every $x,y\in \Rd$ and a.s. in $\omega$,
\begin{equation}\label{aslimsup}
\limsup_{j \to \infty} \left( w^{\delta_j}(y,\omega) -w^{\delta_j}(x,\omega) \right) \leq m_\mu(y,x,\omega) - p\cdot (y-x).
\end{equation}
From this it is easy to obtain \eqref{clm} by passing to limits. Indeed, fixing $z\in \Rd$ and a nonnegative test function $\varphi\in L^1(\Rd \times\Omega)$, we find that, for every $z\in \Rd$,
\begin{equation*}
\begin{aligned}
\lefteqn{\int_{\Rd \times \Omega} \left( w(y,\omega) - w(y-z,\omega) \right) \varphi(y,\omega) \, dy\, d\Prob(\omega)} \qquad & \\ & =  \lim_{j\to \infty} \int_{\Rd \times \Omega} \left( w^{\delta_j}(y,\omega) - w^{\delta_j}(y-z,\omega) \right) \varphi(y,\omega) \, dy\, d\Prob(\omega)\\
& \leq \int_{\Rd\times\Omega} \left( m_\mu(y,y-z,\omega) - p\cdot z \right) \varphi(y,\omega) \, dy \, d\Prob(\omega), 
\end{aligned}
\end{equation*}
where the first passage to the limit is via \eqref{weakstar} and the inequality is justified by the dominated convergence theorem and \eqref{aslimsup}, since the uniform Lipschitz continuity of $w^{\delta_j}$ implies that the integrand is dominated by $C|z|\varphi \in L^1(\Rd\times\Omega)$. We deduce that, for every $z\in \Rd$ and a.e. $(y,\omega) \in \Rd\times \Omega$,
\begin{equation*}\label{}
w(y,\omega) - w(y-z,\omega) + p\cdot z \leq m_{\mu}(y,y-z,\omega).
\end{equation*}
Taking $z=y-x$ and using the continuity of $w$, we obtain that, for every $x,y\in\Rd$ and a.s. in $\omega$,
\begin{equation*}
w(y,\omega) - w(x,\omega) + p\cdot (y-x) \leq m_{\mu}(y,x,\omega).
\end{equation*}
Sending $\mu \ssearrow \overline H(p)$ in view of Remark~\ref{mucontin} yields \eqref{clm}. 

That $w$ has stationary and mean-zero increments is immediate from the analogous property of the $w^\delta$'s and~\eqref{weakstar}, and so we omit the details. Finally, an application of Lemma~\ref{kozlov} yields that $w(\cdot,\omega) \in \SL$ a.s. in $\omega$. 
\end{proof}

When $H$ is convex, the proof of the existence of a subcorrector which is strictly sublinear at infinity is considerably simpler because it is possible to interchange the weak limits and the nonlinear convex $H$ without knowing in advance that there is homogenization. This argument fails, however, for quasiconvex Hamiltonians. It is possible to give a different proof for~Proposition~\ref{usual}, following~\cite{LS3}, provided we use the fact that the equation homogenizes. In the final step we apply Mazur's lemma, which allows us to take linear convex combinations of the weakly convergent sequence to obtain a strongly convergence sequence with the same limit. The level-set convexity of $H$ and the fact that $-\delta_j v^{\delta_j}(y,\omega\,;p)$ converges to $\overline H(p)$ a.s. in~$\omega$ and uniformly on balls of radius$\sim1/\delta$ allows us to pass the limit of the linear convex combinations of the gradient inside $H$.

\subsection*{Acknowledgements}
The first author was partially supported by NSF Grant DMS-1004645 and the second author by NSF Grant DMS-0901802.

\bibliographystyle{plain}
\bibliography{quasiconvex}

\end{document}